\documentclass[11pt]{amsart}
  \baselineskip 1 mm
\usepackage{hyperref}
\hypersetup{
    colorlinks=false, 
    linktoc=all,     
    linkcolor=blue,
        citecolor=red,
    filecolor=black,
    urlcolor=green
}
\usepackage[english]{babel}
\usepackage{bbm}
\usepackage{amssymb}
\usepackage[all]{hypcap}
\usepackage{seqsplit}
\usepackage{bm}
\usepackage{varwidth}
\usepackage{parskip}
\usepackage{enumerate}
\usepackage[inline]{enumitem}
\usepackage{graphicx}
\makeatletter
\newcommand{\inlineitem}[1][]{%
\ifnum\enit@type=\tw@
    {\descriptionlabel{#1}}
  \hspace{\labelsep}%
\else
  \ifnum\enit@type=\z@
       \refstepcounter{\@listctr}\fi
    \quad\@itemlabel\hspace{\labelsep}%
\fi} \makeatother
\newcommand{\ga}{\alpha}
\newcommand{\gb}{\beta}
\newcommand{\gga}{\gamma}
\newcommand{\gd}{\delta}
\newcommand{\gep}{\epsilon}

\newcommand{\gl}{\lambda}
\newcommand{\gm}{\mu}

\newcommand{\gp}{\pi}

\newcommand{\subs}{\subset}
\newcommand{\sups}{\supset}

\newcommand{\bs}{\backslash}

\newcommand{\ti}{\tilde}

\newcommand{\mbb}{\mathbb}

\newcommand{\mcl}{\mathcal}

\newcommand{\us}{\underset}
\newcommand{\os}{\overset}
\newcommand{\Lra}{\Leftrightarrow}

\newcommand{\lra}{\longrightarrow}
\newcommand{\llra}{\longleftrightarrow}

\newcommand{\Z}{\mbb Z}

\newcommand{\La}{\Leftarrow}
\newcommand{\ra}{\rightarrow}
\newcommand{\Ra}{\Rightarrow}

\newcommand{\equ}[1]{%
\begin{equation*}
#1
\end{equation*}
}
\newcommand{\equa}[1]{%
\begin{equation*}
\begin{aligned}
#1
\end{aligned}
\end{equation*}
}
\newcommand{\equan}[2]{%
\begin{equation}
\label{Eq:#1}
\begin{aligned}
#2
\end{aligned}
\end{equation}
}

\newcommand{\mattwo}[4]{%
\begin{pmatrix}
  #1 & #2\\ #3 & #4
\end{pmatrix}
}

\newcommand{\matcolfive}[5]{%
\begin{pmatrix}
  #1\\#2\\#3\\#4\\#5
\end{pmatrix}
}

\newcommand{\mattwofour}[8]{%
\begin{pmatrix}
  #1 & #2 & #3 & #4\\#5 & #6 & #7 & #8
\end{pmatrix}
}

\newcommand{\mattwosix}[9]{%
  \def\argi{{#1}}%
  \def\argii{{#2}}%
  \def\argiii{{#3}}%
  \def\argiv{{#4}}%
  \def\argv{{#5}}%
  \def\argvi{{#6}}%
  \def\argvii{{#7}}%
  \def\argviii{{#8}}%
  \def\argix{{#9}}%
  \mattwosixRelay
}
\newcommand\mattwosixRelay[3]{%
\begin{pmatrix}
  \argi & \argii & \argiii & \argiv & \argv & \argvi\\
  \argvii & \argviii & \argix & #1 & #2 & #3
\end{pmatrix}
}

\newtheorem{thm}{Theorem}[section]
\newtheorem{theorem}[thm]{Theorem}
\newtheorem{lemma}[thm]{Lemma}

\newtheorem{cor}[thm]{Corollary}

\theoremstyle{definition}
\newtheorem{defn}[thm]{Definition}

\newtheorem{ques}[thm]{Question}
\newtheorem{claim}[thm]{Claim}

\newtheorem{example}[thm]{Example}
\newtheorem{note}[thm]{Note}

\numberwithin{equation}{section}
\begin{document}

\title[On the Geometry of Multiplicatively Closed Sets with
Arbitrarily Large Gaps]{On the Geometry of Multiplicatively Closed Sets Generated by at most Two 
Elements with Arbitrarily Large Gaps, a Constructive Method}
\author[C.P. Anil Kumar]{C.P. Anil Kumar}
\address{Stat Math Unit, Indian Statistical Institute,
  8th Mile Mysore Road, Bengaluru-560059, India}
\email{akcp1728@gmail.com}

\thanks{The author would like to thank Indian Statistical Institute (ISI) Bengaluru, India 
for the support. The work was done while the author was a visiting scientist at ISI}
\maketitle
\begin{abstract}
We prove in Theorem~\ref{theorem:GapsDoublyGenerated} that the multiplicatively closed subset
generated by at most two elements in the set of natural numbers $\mbb{N}$ has arbitrarily large 
gaps by explicitly constructing large integer intervals with known prime factorization for
the end points, which do not contain any element from the multiplicatively closed set 
apart from the end points, which belong to the multiplicatively closed set. 
An Example~\ref{example:Main} is also illustrated.  

We also give a criterion in 
Theorems~\ref{theorem:DoubleGen},~\ref{theorem:LineClassification} by using a geometric 
correspondence between maximal singly generated multiplicatively closed sets and points of the 
space $\mbb{PF}^{\infty}_{\mbb{Q}\geq 0}$ (refer to Theorem~\ref{theorem:Bijection}) as to when 
a finitely generated multiplicatively closed set gives rise to a doubly multiplicatively closed 
line (refer to Definition~\ref{defn:DoublyMultClosedLine}). We answer a similar 
Question~\ref{ques:Gaps} partially about gaps in a multiply-generated multiplicative closed set, 
when it is contained in a doubly multiplicative closed set using Theorem~\ref{theorem:DoubleGen}
and Theorem~\ref{theorem:MultiplyGeneratedDoublyClosedLine}. 

In the appendix Section~\ref{sec:Appendix} we discuss another constructive proof 
(refer to Theorem~\ref{theorem:AnotherConstructiveProof}) for arbitrarily large gap
intervals, where the prime factorization is not known for the right end-point unlike the 
constructive proof of the main result of the article in the case of multiplicatively closed set
$\{p_1^ip_2^j\mid i,j \in \mbb{N}\cup\{0\}\}$ with $p_1<p_2,Log_{p_1}(p_2)$ irrational 
for which the prime factorization is known for both the end-points of the gap interval via the 
stabilization sequence of the irrational $\frac{1}{Log_{p_1}(p_2)}$.
\end{abstract}
\bigskip
\section{\bf{Introduction}}
Historically we have seen more of existence proofs 
of arbitrary large gaps in certain subsets of integers that are present in the literature.
A short survey below mentions such results. 
However constructive proofs in particular those which give the formulae for the end points of the 
arbitrary large gap intervals have not been there. Here in this article we will be interested in 
one such constructive proof.
\subsection{Short Survey}
\label{sec:History}
The distribution of integers with exactly $k-$ distinct prime factors has been 
studied by many authors. It was first shown by Landau~\cite{LE} that for a fixed
$k \geq 1$, the function defined by
\equ{\gp(x,k)=\us{n \leq x}{\sum} f_k(n),}
where $f_k(n)=1$ if $n$ has exactly $k$-prime factors and $0$ otherwise satisfies
\equan{LargeGap}
{\gp(x,k) = \bigg(\frac{x}{log\ x}\bigg)\frac{(log\ log\ x)^{k-1}}{(k-1)!}(1+o(1)).}
Among the other authors who have obtained similar or better asymptotic expressions
are Sathe~\cite{LGSIII,LGSIIIIV}, Selberg~\cite{SA},Hensely~\cite{HD},Hildebrand and 
Tenenbaum~\cite{HATG}.

Let $\{p_1,p_2,\ldots,p_k\}$ be any set of $k-$distinct primes. Let 
$\mbb{S}_{\{p_1,p_2,\ldots,p_k\}}$ be the multiplicatively closed set generated by $1$
and numbers, which have exactly and all the factors from $\{p_1,p_2,\ldots,p_k\}$. Let $\mcl{C}$ 
be the collection of all $k-$subsets of prime numbers. Consider the set
\equ{\mbb{S}_k=\us{c \in \mcl{C}}{\bigcup}\mbb{S}_{c}.}
Using any of the results say the result by Landau~\cite{LE} about asymptotics of $\gp(x,k)$ we 
conclude that there are arbitrarily large gaps
in $\mbb{S}$. We observe here that using Equation~\ref{Eq:LargeGap} we have
\equ{\us{x \lra \infty}{lim}\frac{\gp(x,k)}{x}=0.}
If the gaps were bounded then we have that
\equ{\us{x \lra \infty}{liminf}\frac{\gp(x,k)}{x}>0}
would be a non-zero constant. Hence the gaps must be arbitrarily large in the set $\mbb{S}_k$.

With an additional bit of effort on the result of Landau~\cite{LE} we can extend and 
conclude arbitrary large gaps for the set
\equ{\us{i=1}{\os{k}{\bigcup}}\mbb{S}_i.}

Now choose a base say $b=2$. If we use asymptotics for a multiplicatively closed set $\mbb{T}$ 
generated by primes $\{p_1,p_2,\ldots,p_k\}$ then we get for large $x$ the following inequality
\equ{\lceil\frac{log_b\ x}{\us{i=1}
{\os{k}{\sum}}log_b\ p_i}\rceil \leq  \#(\mbb{T}\cap [1,x]) \leq \us{i=1}{\os{k}{\prod}}
\lceil\frac{log_b\ x}{log_b\ p_i}\rceil.}
Hence again we have
\equ{\us{x \lra \infty}{lim}\frac{\#(\mbb{T}\cap [1,x])}{x}=0}
from which we will be able to conclude that there are arbitrarily large gaps in $\mbb{T}$.

However here in this article we give a constructive proof for multiplicatively closed 
sets, which are contained in doubly generated multipicatively closed sets with known
generators. First we consider
multiplicatively closed sets generated by two primes or more generally two positive integers ($>1$), which are 
not Log-Rational to each other. We note here that the multiplicatively closed set can contain 
numbers with single prime factor unlike the set, which is considered in the result by 
Landau~\cite{LE}. Using the technique of rational approximation and stabilization of the 
sequences of approximate inverses and increasing gaps between two such successive ones we 
explicitly construct by locating large intervals of natural numbers, which do not contain any 
element in the given multiplicatively closed set there by proving the main result
Theorem~\ref{theorem:GapsDoublyGenerated} given below.
\section{\bf{The main result and method of Proof}}
Let $\mbb{N}$ denote the set of natural numbers. 
Let $\mbb{P}=\{2,3,5,\ldots,\}$ denote the set of primes.  
Here we give using techniques from number theory, geometry 
and finite fields, a constructive proof of the main result~\ref{theorem:GapsDoublyGenerated}, 
where the prime factorization of the end-points of the 
gap intervals are known and also the end-points belong to the multiplicatively closed set itself.

Before we state the main result we need a definition.
\begin{defn}[Stabilization sequence of an irrational using sequences of approximate inverses]
Let $0 < \ga < 1$ be an irrational. Let $\frac{p_n}{q_n},gcd(p_n,q_n)=1,p_n<q_n$ be any sequence of
positive rationals converging to $\ga$. Now consider the arithmetic progressions
$p_n\Z^{+}$ and $q_n\Z^{+}$. Consider the sequence \equ{\big((p_n\Z^{+} \cup
q_n\Z^{+} \cup \{0\})\cap \{0,1,2,\ldots,q_np_n\}\big)\subs\{0,1,2,\ldots,q_np_n\}}
given as follows.
\equa{&l_0(n)=0,p_n,2p_n,3p_n,\ldots,l_1(n)p_n,q_n,\\
&(l_1(n)+1)p_n,(l_1(n)+2)p_n,\ldots,l_2(n)p_n,2q_n,\\
&(l_2(n)+1)p_n,\ldots,l_i(n)p_n,iq_n,\\
&(l_i(n)+1)p_n,\ldots,(q_n-1)p_n,p_nq_n.} 
For every $n\in \mbb{N}$, define the sequence of numbers
\equ{\{l_{j_1(n)}(n),l_{j_2(n)}(n),\ldots,l_{j_{r_n}(n)}(n)\}}
given as follows. We define $j_1(n)=0,l_0(n)=0$. Now let 
$j(n)\in \{j_1(n),j_2(n),\ldots, j_{r_n}(n)\}$.

The defining/characterizing property for $l_{j(n)}(n)$ is given by
\equ{\boxed{q_n \geq j(n)\geq 1, 
(l_{j(n)}(n)+1)p_n-j(n)q_n < min\{(l_i(n)+1)p_n-iq_n\mid 0\leq i<j(n)\}.}}
We have 
$\{l_{j_1(n)}(n),l_{j_2(n)}(n),\ldots,l_{j_{r_n}(n)}(n)\}$=\begin{center}
$
\begin{cases}
=\{0=l_0(n)=l_{j_1(n)}(n)=p_n^{-1}-1\mod\ q_n\} \text{ if } p_n=1\\
=\{0=l_{j_1(n)}(n)=l_0(n) < l_{j_2(n)}(n) = l_1(n) < l_{j_3(n)}(n)< \ldots \\ <
(l_{j_{r_n}(n)}(n)=p_n^{-1}-1\mod\ q_n)\}  \text{ if } p_n \neq 1.
\end{cases}
$
\end{center}
The sequence \equ{\{1=l_0+1=l_{j_1(n)}(n)+1,l_{j_2(n)}(n)+1,\ldots,l_{j_{r_n}(n)}(n)+1\}} 
is the sequence of approximate inverses of $p_n\mod q_n$.
By using Theorems~\ref{theorem:IncreasingGaps},\ref{theorem:StabInvariance},\ref{theorem:Gaps}
we conclude that the gaps 
\begin{itemize}
\item $l_{j_{i+1}(n)}(n)-l_{j_i(n)}(n)$ in the above sequence is
increasing. 
\item The values $j_i(n)$ stabilize and also $l_{j_i(n)}(n)$
is eventually a constant as $n \lra \infty$ for a stabilized $j_i$. 
(Let the stabilized constant be denoted by $l_{j_i}$)
\item We have $\us{i \lra \infty}{lim}(l_{j_{i+1}}-l_{j_{i}})\nearrow\infty$.
\end{itemize}
This stabilized approximate inverse sequence $\{l_{j_i}+1:i\in \mbb{N}\}$ is called the
stabilization sequence of the irrational $\ga$. 
\end{defn}
Now we state the main result.
\begin{theorem}
\label{theorem:GapsDoublyGenerated}
Let $\mbb{PP}=\{p_1<p_2\}$ be set a set of two integers, which are not log-rational to each other. 
Let $\ga=\frac{1}{Log_{p_1}(p_2)}$ be the associated irrational less than one. 
Let $\{s_i:i\in \mbb{N}\}$ be the stabilization sequence of 
$\ga$. Let $t_i=\lfloor s_i\ga\rfloor$. Then 
\begin{enumerate}
\item The integer interval 
\equ{(p_2^{t_{i+1}-1},\ldots,p_1^{s_i}p_2^{t_{i+1}-t_i-1})}
contains no element in the multiplicatively closed set generated by $\mbb{PP}$
apart from the end-points. 
\item The limit of the above gaps (length of the above interval) tends to infinity
better than a geometric progression with common ratio $p_2$.
\end{enumerate}
\end{theorem}
This theorem is illustrated with the Example~\ref{example:Main}.
As an application of Theorem~\ref{theorem:GapsDoublyGenerated} we have the following corollary.
\begin{cor}
\label{cor:GapsDoublyGenerated}
Let $\mbb{A}$ be a finite set of positive natural numbers. Let $\mbb{PP} \neq \{1\}$ be a nonempty set of at most two natural numbers. Let $\mbb{S}=\{1<a_1<a_2<\ldots<a_n<\ldots\} \subs
\mbb{N}$ be the infinite multiplicatively closed set generated by $\mbb{A}$.
Suppose the multiplicatively closed set $\mbb{S}\subs \langle \mbb{PP} \rangle$ 
multiplicatively closed set generated by the set $\mbb{PP}$. Then we have 
\begin{itemize}
\item $\us{n \lra \infty}{limsup} (a_{n+1}-a_{n}) = \infty.$
\item We have explicit expressions for the end points of certain arbitrarily large gap intervals
in the set $\mbb{S}$ using the generators of $\mbb{PP}$.
\end{itemize}
\end{cor}


\subsection{Summary of the Proof}
~\\
We summarize the method of proof and the structure of the paper in this section.

In Section~\ref{sec:IrrBehRatApp} we first show that for any two relatively prime numbers $1<p<q$
the gaps between successive approximate inverses of $p\mod q$ is increasing in 
Theorem~\ref{theorem:IncreasingGaps}. In 
Theorems~\ref{theorem:StabInvariance},~\ref{theorem:Gaps} we prove for a sequence
of positive rationals converging to an irrational in $[0,1]$, the sequence of approximate 
inverses eventually stabilize and the gaps between successive approximate
inverses increase.

In Section~\ref{sec:TheoremGaps} we prove our main Theorem~\ref{theorem:GapsDoublyGenerated}. We 
consider a multiplicatively closed set $\mbb{S}$ generated by two positive numbers $p_1,p_2$,
which are log irrational to each other i.e. $Log_{p_1}(p_2)$ is irrational.
We apply
Theorems~\ref{theorem:StabInvariance},~\ref{theorem:Gaps} to $\frac{1}{Log_{p_1}(p_2)}$
for a suitable sequence of positive rationals obtained in
Lemma~\ref{lemma:IncreasingLemma} and conclude increasing gaps for the stabilized sequence. Then 
we locate integer intervals in Lemma~\ref{lemma:EmptyLemma} of arbitrarily large size, which has 
no elements from the multiplicatively closed set $\mbb{S}$. This finally proves our main 
Theorem~\ref{theorem:GapsDoublyGenerated} and Corollary~\ref{cor:GapsDoublyGenerated}.

This Theorem~\ref{theorem:GapsDoublyGenerated} leads to an open Question~\ref{ques:Gaps}.
We state this open question in Section~\ref{sec:OQ}.

In an attempt to answer this open Question~\ref{ques:Gaps}, in Section~\ref{sec:GMTG}, we 
mention that a generalization of the proof of 
Theorem~\ref{theorem:GapsDoublyGenerated} is not directly feasible by proving 
Lemma~\ref{lemma:235} via an example.

In Section~\ref{sec:GeometryMultClosed} we associate to every multiplicatively closed set a 
point in the projective space $\mbb{PF}^{\infty}_{\mbb{Q}_{\geq 0}}$
and conversely to every point, a maximal singly generated multiplicatively closed set in 
Theorem~\ref{theorem:Bijection}. Then we characterize when two points
$P_1,P_2\in \mbb{PF}^{\infty}_{\mbb{Q}_{\geq 0}}$ give rise to the same point in terms of 
Log-Rationality in Theorem~\ref{theorem:Equivalence}. In Theorem~\ref{theorem:DoubleGen} we 
give a criterion for when a finitely generated multiplicatively closed set is contained in doubly 
generated multiplicatively closed set and in Theorem~\ref{theorem:LineClassification} we 
classify doubly multiplicatively closed lines(refer to 
Definition~\ref{defn:DoublyMultClosedLine}).

In view of Question~\ref{ques:Gaps}, if a multiplicatively closed set $\mbb{S}$ is generated by 
$r-$ elements and these generators give rise to $s-$distinct points in the projective space 
$\mbb{PF}^{\infty}_{\mbb{Q}_{\geq 0}}$ (refer to Definition~\ref{defn:ProjSpaceMult}) with 
$s \leq r$ then $\mbb{S}$ is contained in a multiplicatively closed set $\mbb{T}$, which is 
generated by $s-$elements. So Theorem~\ref{theorem:GapsDoublyGenerated} can be used to answer 
Question~\ref{ques:Gaps} whenever $s\leq 2$ with a known single generator or pair of generators 
in the affirmative using the same construction
(refer to Section~\ref{sec:TheoremGaps}). Even otherwise also, if these $s-$points
generate a doubly multiplicative closed line (refer to Definition~\ref{defn:DoublyMultClosedLine}
and Theorems~\ref{theorem:DoubleGen},\ref{theorem:LineClassification})
then Theorem~\ref{theorem:GapsDoublyGenerated} can be used to answer Question~\ref{ques:Gaps}
in the affirmative using the same construction(refer to 
Theorem~\ref{theorem:MultiplyGeneratedDoublyClosedLine}).

In Section~\ref{sec:GeometryMultClosed}, Theorem~\ref{theorem:DoubleGen} and 
Example~\ref{example:LineNotMultClosed} leads to the following
interesting question, which is answered completely in Theorem~\ref{theorem:LineClassification}.
\begin{ques}
\label{ques:ClassifyLines}
Classify all lines $L$ obtained by joining two points $P_1,P_2 \in 
\mbb{PF}^{\infty}_{\mbb{Q}_{\geq 0}} \subs \mbb{PF}^{\infty}_{\mbb{Q}}$, which are 
doubly multiplicatively closed lines.
\end{ques}

\section{\bf{Irrationals and behaviour of rational approximations, arithmetic
progressions,stabilization}} \label{sec:IrrBehRatApp}
We start this section by proving a theorem below on increasing gaps for the successive 
approximate inverses.
\begin{theorem}[Increasing gaps between successive approximate inverses]
\label{theorem:IncreasingGaps} Let $p,q$ be two positive integers
with $gcd(p,q)=1,p<q$. Consider the arithmetic progressions
$p\Z^{+}$ and $q\Z^{+}$. Consider the sequence $(p\Z^{+} \cup
q\Z^{+} \cup \{0\})\cap \{0,1,2,\ldots,qp\}$ in the set
$\{0,1,2,\ldots,qp\}$.
\equa{&l_0=0,p,2p,3p,\ldots,l_1p,q,\\
&(l_1+1)p,(l_1+2)p,\ldots,l_2p,2q,\\
&(l_2+1)p,\ldots,l_ip,iq,\\
&(l_i+1)p,\ldots,(q-1)p,qp.} Now consider the sequence of numbers
\equa{&\{l_0=0\} \cup \{l_j \mid q \geq j\geq 1,
(l_j+1)p-jq < \us{0\leq i<j}{min}\{(l_i+1)p-iq\}\}\\&
=\{l_{j_1},l_{j_2},\ldots,l_{j_r}\}}
\begin{center}
$
\begin{cases}
=\{0=l_0=l_{j_1}=p^{-1}-1\mod\ q\} \text{ if } p=1\\
=\{0=l_{j_1}=l_0 < l_{j_2} = l_1 < l_{j_3}< \ldots <
(l_{j_r}=p^{-1}-1\mod\ q)\}  \text{ if } p \neq 1.
\end{cases}
$
\end{center}
Then the gaps $l_{j_{i+1}}-l_{j_i}$ in the above sequence is
increasing.
\end{theorem}
\begin{proof}
If $p=1$ then there is nothing to prove. So assume $p>1$.
First we
observe that $p$ is a unit in $\Z/q\Z=\{0,1,2\ldots,q-1\}$. The
values $(l_i+1)$ tend to the inverse of $p$ because the least
possible value for $(l_i+1)p-iq$ is one. If we consider the sequence
of multiples $\{(l_{j_1}+1)p\ mod\ q,(l_{j_2}+1)p\ mod\
q,\ldots,(l_{j_r}+1)p\ mod\ q\}$
then the values are distinct and decrease to $1$ as multiplies of $p$ given by
$0,p,2p,\ldots,(q-1)p$ gives rise to all residue classes modulo $q$.
Now suppose we consider three consecutive elements in the sequence 
$l_{j_i},l_{j_{i+1}},l_{j_{i+2}}$ then we have
\equa{(l_{j_i}+1)p &= k_{j_i}q+x_{j_i}\\
(l_{j_{i+1}}+1)p &= k_{j_{i+1}}q+x_{j_{i+1}}\\
(l_{j_{i+2}}+1)p &= k_{j_{i+2}}q+x_{j_{i+2}}\\
\text{ and the residue classes satisfy }&
x_{j_i}>x_{j_{i+1}}>x_{j_{i+2}}} and moreover for any $t <
l_{j_{i+1}}-l_{j_i}$ we have \equ{\text{ if } (l_{j_i}+1+t)p=kq+x
\text{ then } x > x_{j_{i}}} because of the minimality condition on
$(l_{j_i}+1)p-j_iq$ as the lesser than $(l_{j_i}+1)$ multiples of
$p$ are not as close to multiples of $q$, where we compare multiples
of $p$ to numbers, which are smaller and multiples of $q$. So we have
\equ{(l_{j_{i+1}}+1+t)p=(l_{j_{i+1}}-l_{j_i})p+(l_{j_i}+1+t)p=
(k_{j_{i+1}}-k_{j_i}+k)q+x_{j_{i+1}}-x_{j_i}+x.}
Now note in the right hand side we have the following inequalities
for the residue classes $mod\ q$.
\equa{&0< x_{j_i} < q\\
&0< x_{j_{i+1}} < q\\
&0 < x_{j_i} - x_{j_{i+1}} < q\\
&0<x_{j_{i+1}}< x_{j_{i+1}}-x_{j_i}+x < x < q} This is a subtle
argument about the residue classes. Hence we have $l_{j_{i+2}} >
l_{j_{i+1}}+t$ for all $t < l_{j_{i+1}}-l_{j_i}$ and for
$t=l_{j_{i+1}}-l_{j_i}$ we have $x=x_{j_{i+1}}$ so a candidate for
the residue class is $(2x_{j_{i+1}}-x_{j_i})$ and
\equ{(l_{j_{i+1}}+1+t)p=(2k_{j_{i+1}}-k_{j_i})q+
(2x_{j_{i+1}}-x_{j_i}).} So we have if $0< (2x_{j_{i+1}}-x_{j_i})$
then the residue class is $(2x_{j_{i+1}}-x_{j_i})$ and \equ{0<
(2x_{j_{i+1}}-x_{j_i}) = x_{j_{i+1}}+x_{j_{i+1}}-x_{j_i} <
x_{j_{i+1}}<q.} So $l_{j_{i+2}}=2l_{j_{i+1}}-l_{j_i}$ or
$l_{j_{i+2}}-l_{j_{i+1}}=l_{j_{i+1}}-l_{j_i}$. Otherwise if $0 <
x_{j_{i+1}}< 2x_{j_{i+1}} < x_{j_i}<q$ then the residue class is
given by $q+2x_{j_{i+1}}-x_{j_i}$ and we observe
that\equ{q>q+2x_{j_{i+1}}-x_{j_i}>x_{j_{i+1}} \text{ because } q >
q+x_{j_{i+1}}-x_{j_i}>0} we conclude that $l_{j_{i+2}} >
2l_{j_{i+1}}-l_{j_i}$ or
$l_{j_{i+2}}-l_{j_{i+1}}>l_{j_{i+1}}-l_{j_i}$. It is also clear that
the residue classes decrease to one.
Now Theorem~\ref{theorem:IncreasingGaps} follows.
The sequence \equ{\{1=l_0+1=l_{j_1}+1,l_{j_2}+1,\ldots,l_{j_r}+1\}} is the sequence of 
approximate inverses of $p\mod q$.
\end{proof}
The following theorem is a stabilization theorem for approximate inverses for a 
converging sequence of rationals. 
\begin{theorem}[Stabilization and Eventual Invariance]
\label{theorem:StabInvariance}
Let $p_n,q_n$ be a sequence of positive integers with
$gcd(p_n,q_n)=1$ and suppose $\frac{p_n}{q_n}$ is a cauchy sequence
converging to an irrational number $0<\ga<1$. Define as in the
previous lemma the sequence $l_i(n)$ and consider the set
\equa{\{0=l_{j_1(n)}(n)<l_{j_2(n)}(n)=l_1(n)<l_{j_3(n)}(n)<\ldots&<l_{j_{r_n}(n)}(n)\\
&=p_n^{-1}-1\mod\ q_n\}.}
The values $j_i(n)$ stabilize and also $l_{j_i(n)}(n)$
is eventually a constant as $n \lra \infty$ for a stabilized $j_i$.
\end{theorem}
\begin{proof}
We can assume that $p_n < q_n$ and $p_n \neq 1$. If $p_n=1$ for
infinitely many positive integer $n>0$ then $\frac{p_n}{q_n} \lra 0$, which is a contradiction.
We observe that $l_i(n)=\lfloor\frac{iq_n}{p_n}\rfloor$ and for
fixed $i,l_i(n)$ is eventually $\lfloor \frac{i}{\ga} \rfloor$ as $n
\lra \infty$. Also we have the sequence $j_i(n)$ stabilizes as $n
\lra \infty$ because in the inductive definition, we have $j_i(n)$
satisfies the property that \equ{(l_{j_i(n)}(n)+1)p_n-j_i(n)q_n <
\us{0 \leq i < j_i(n)}{min}\{(l_i(n)+1)p_n-iq_n\}} or equivalently
that \equ{(l_{j_i(n)}(n)+1)\frac{p_n}{q_n}-j_i(n)<\us{0 \leq i <
j_i(n)}{min}\{(l_i(n)+1)\frac{p_n}{q_n}-i\}.} Now if $n \lra \infty$
then we get that $(l_i(n)+1)\frac{p_n}{q_n}-i \lra \big(\lfloor
\frac{i}{\ga}\rfloor+1\big)\ga-i$, which is independent of $n$. Now the independence of $n$ here
implies the stabilization of $j_i(n)$ follows as $n \lra \infty$. This
completes the proof of Theorem~\ref{theorem:StabInvariance}.
\end{proof}
The theorem below along with Weyl Equidistributive Criterion, 
Theorem~\ref{theorem:WeylEquiCrit} establishes the increasing nature of 
gaps in the stabilized approximate inverses for a converging sequence of 
rationals to an irrational number.
\begin{theorem}
\label{theorem:Gaps} Let $p_n,q_n$ be a sequence of positive
integers with $gcd(p_n,q_n)=1$ with $p_n<q_n$ and suppose $\frac{p_n}{q_n}$ is a
cauchy sequence converging to an irrational number $0 < \ga < 1$. Define as
in the previous lemma the sequence $l_i(n)$ and consider the set
\equa{\{0=l_{j_1(n)}(n)<l_{j_2(n)}(n)=l_1(n)<l_{j_3(n)}(n)<\ldots&<l_{j_{r_n}(n)}(n)\\
&=p_n^{-1}-1\mod\ q_n\}.} Using the previous lemma let $j_i=\us{n \lra \infty
}{lim}j_i(n),l_i=\us{n \lra \infty }{lim}l_i(n)$. Then we have
\equ{\us{i \lra \infty}{lim}l_{j_{i+1}}-l_{j_{i}}=\infty.}
\end{theorem}
\begin{proof}
We can assume that $p_n \neq 1$ eventually.
We observe that using the previous Theorem\ref{theorem:StabInvariance}
we have for every \equ{i \in \mbb{N}, l_{j_{i+2}}-l_{j_{i+1}} \geq l_{j_{i+1}}-l_{j_{i}}.}
If the above limit is not infinity (say equal to $d$) then eventually $l_{j_i}$ form an 
arithmetic progression with common difference $d$. Then
$(l_{j_i}+1)=\lceil \frac{j_i}{\ga}\rceil$ is in  arithmetic progression with common difference 
$d$. On the one hand the sequence
\equ{\lceil \frac{j_i}{\ga}\rceil \ga - j_i \searrow 0} On the other hand the sequence has a 
distribution if $l_{j_i}$ are in arithemetic progression.
Because if $l_{j_i}=l_{j_{i_0}}+kd$ with $k \in \mbb{N}$ and fractional parts $z_{j_i}$ are 
such that $\frac{j_i}{\ga}+ z_{j_i} = \lceil \frac{j_i}{\ga}\rceil = l_{j_i}+1$. Then we get 
$(l_{j_{i_0}}+kd+1)\ga-j_i = z_{j_i}\ga \searrow 0$. However the fractional
parts $\{(l_{j_{i_0}}+kd+1)\ga-j_i\}=\{(l_{j_{i_0}}+kd+1)\ga\}$ are distributed in the unit 
interval uniformly as $k \in \mbb{N}$ by Weyl's Criterion, Theorem~\ref{theorem:WeylEquiCrit}.
So this is a contradiction and Theorem~\ref{theorem:Gaps} follows.
\end{proof}
We mention Weyl's Equidistributive Criterion here (See also~\cite{HWEYL}.)
\begin{theorem}
\label{theorem:WeylEquiCrit}
Let $\ga$ be a positive irrational. Let $0 \leq  a \leq b \leq 1$. For $x \in \mbb{R}^{+}$, 
let $\{x\}$ denote the fractional part of $x$.  Then we have
\equ{\frac{\#\{n\mid a \leq \{n\ga\} \leq b, 1 \leq n \leq N\}}{N} \lra (b-a) 
\text{ as }N\lra \infty.}
\end{theorem}
\section{\bf{The main theorem and construction of arbitrarily large gaps}}
\label{sec:TheoremGaps}
Before we prove the main Theorem~\ref{theorem:GapsDoublyGenerated} 
of this section we prove the following three 
Lemmas~\ref{lemma:Irrational},~\ref{lemma:IncreasingLemma},~\ref{lemma:EmptyLemma}.
\begin{lemma}
\label{lemma:Irrational}
Let $p_1<p_2$ be two natural numbers such that $gcd(p_1,p_2)=1$. Then
\begin{itemize}
\item Either $p_1=1$.
\item Or $Log_{p_1}(p_2),Log_{p_2}(p_1)$ are both simultaneously irrationals.
\end{itemize}
\end{lemma}
\begin{proof}
If $p_1=1$ then there is nothing to prove.
Suppose $Log_{p_1}(p_2)=\frac mn$ for some positive integers $m,n>0$.
Then we have $p_2^n=p_1^m$ a contradiction to unique factorization into primes.
So $Log_{p_1}(p_2)$ is irrational.
\end{proof}
\begin{defn}
We say a pair $(p_1,p_2)\in \mbb{N}^2$ is an irrational pair if $p_1\neq 1$ and $p_2\neq 1$ and 
both $Log_{p_1}(p_2),Log_{p_2}(p_1)$ are irrationals. For example a $GCD-$one pair 
$(p_1,p_2) \in \mbb{N}^2,$
where $p_1 \neq 1 \neq p_2$ is an irrational pair.
\end{defn}
\begin{lemma}
\label{lemma:IncreasingLemma}
Let $(p_1,p_2)\in \mbb{N}^2$ be such that $p_1<p_2$ and is an irrational pair.
Let $\ga=Log_{p_2}(p_1)<1$. Let $x_2(i)=\lceil\frac{i}{\ga}\rceil$. For
every positive integer $i$ let \equ{z_i=-i+x_2(i)\ga.} Define
a subsequence with the property that
\equ{z_{k_j}<z_{k_{j-1}}=min\{z_1,z_2,\ldots,z_{k_j-1}\}.} Then
\begin{enumerate}
\item $z_{k_j} \searrow 0$.
\item $k_j-k_{j-1}$ is increasing.
\item $\us{j \lra \infty}{lim}
(k_j-k_{j-1}) = \infty.$
\end{enumerate}
\end{lemma}
\begin{proof}
First we define a sequence of number parts $0<y_i<1$ defined by the equation 
\equ{y_i+\frac{i}{\ga}=\lceil\frac{i}{\ga}\rceil=x_2(i).}
Define a subsequence with the property that 
\equ{y_{k_j}<y_{k_{j-1}}=min\{y_1,y_2,\ldots,y_{k_j-1}\}.} Since the number parts of
$\{\frac{i}{\ga}\mid i \in \mbb{N}\}$ is also dense in $[0,1]$ we have that $y_{k_j}\searrow 0$.

We also have for every $i$, $z_i=y_i\ga$. So $z_{k_j}$ also satisfies the property that
\equ{z_{k_j}<z_{k_{j-1}}=min\{z_1,z_2,\ldots,z_{k_j-1}\}.}
Now we have $x_2(k_j)=\frac{k_j}{\ga}+y_{k_j}$ and $y_{k_j} \searrow 0$. Since $y_{k_j}\ga<1$ 
then
\equ{\lfloor x_2(k_j)\ga\rfloor = k_j.}
Now we apply the previous Theorems~\ref{theorem:StabInvariance},~\ref{theorem:Gaps} as follows. 
The sequence
\equ{\frac{k_j}{x_2(k_j)}=\ga-\frac{z_{k_j}}{x_2(k_j)} \lra \ga\text{ as }j \lra \infty}
In Theorems~\ref{theorem:StabInvariance},~\ref{theorem:Gaps} we choose $\ga$, which is an 
irrational satisfying the property that
$0 < \ga <1$ and the sequence of rationals $\frac{k_j}{x_2(k_j)}=\frac{p_j}{q_j} \lra \ga$ as 
$j \lra \infty,$ where $gcd(p_j,q_j)=1$.
Now by the very definition of $z_{k_j}$ and using the properties of stabilization and eventual 
invariance we have
\begin{itemize}
\item $x_2(k_j)-x_2(k_{j-1})$ is increasing.
\item $\us{j \lra \infty}{lim}(x_2(k_j)-x_2(k_{j-1}))=\infty$.
\end{itemize}
This implies we also have
\begin{itemize}
\item $k_j-k_{j-1}$ is increasing.
\item $\us{j \lra \infty}{lim}(k_j-k_{j-1})=\infty$.
\end{itemize}
This proves Lemma~\ref{lemma:IncreasingLemma}.
\end{proof}

Now we prove the following Lemma~\ref{lemma:EmptyLemma}
\begin{lemma}
\label{lemma:EmptyLemma}
Let $p_1<p_2$ be two integers such that $(p_1,p_2)$ is an irrational pair. Using the notations 
of the previous Lemma~\ref{lemma:IncreasingLemma}, we have for any integer $0 \leq t < 
k_{j+1}-k_j$ there are no numbers of the form $p_1^bp_2^a$ in the integer interval excluding
the end-points.
\equ{(p_2^{k_j+t},\ldots,p_1^{x_2(k_j)}p_2^t).}
\end{lemma}
\begin{proof}
Let $\ga=Log_{p_2}(p_1)<1$. Here we use the following fact. We have 
$\lfloor x_2(k_j)\ga\rfloor =k_j$.
Suppose if there exists such a number $p_2^{k_j+t}< p_1^bp_2^a< p_1^{x_2(k_j)}p_2^t$ then
we have \equa{&k_j+t < a+b\ga < t + x_2(k_j)\ga < t+k_j+1\\
\ra &k_j<-t+a+b\ga<x_2(k_j)\ga<k_j+1\\
\ra &k_j+t-a<b\ga<k_j+t-a+1.}

So we have that $b \neq 0$. Similarly $b \neq x_2(k_j)$. If $b=x_2(k_j)$ then we get that 
$k_j=k_j+t-a$, which implies $t=a$. Hence $p_1^bp_2^a$ is an end-point, which is not considered.

Let $b\ga=k_j+t-a+z$. Consider the case $k_j+t-a<k_{j+1}$. Then by definition of 
$z_{k_j},z_{k_{j+1}}$ and since $b \neq x_2(k_j)$ we have $z\geq z_{k_j}>z_{k_{j+1}}$. Hence
\equ{k_j<k_j+z=-t+a+b\ga<x_2(k_j)\ga=k_j+z_{k_j}<k_j+1.}
Hence we get $z<z_{k_j}$, which is a contradiction. Hence we must
have $k_j+t-a\geq k_{j+1}$, which implies $t\geq k_{j+1}-k_j+a \geq k_{j+1}-k_j$, which is again
a contradiction to the hypothesis $0 \leq t < k_{j+1}-k_j$. This proves 
Lemma~\ref{lemma:EmptyLemma}.
\end{proof}

Using these three 
Lemmas~\ref{lemma:Irrational},~\ref{lemma:IncreasingLemma},~\ref{lemma:EmptyLemma}
we prove our main Theorem~\ref{theorem:GapsDoublyGenerated} of this article and its 
Corollary~\ref{cor:GapsDoublyGenerated}.
\begin{proof}
Suppose $\mbb{S}=\{1,f,f^2,\ldots\}$ a singly generated multiplicatively closed set then we 
immediately have $\us{j \lra \infty}{lim} (f^{j+1}-f^j)=\infty$.

Now suppose $\mbb{S}=\{g_1^ig_2^j\mid i,j \geq 0\}$ and $Log_{g_1}(g_2)$ is rational then 
$\mbb{S}$ is contained in a singly generated
multiplicatively closed set $\mbb{T}$ using Theorem~\ref{theorem:Equivalence}. So there exists 
arbitrarily large gaps in $\mbb{S}$ as well.

Now suppose $\mbb{S}=\{p_1^ip_2^j\mid i,j\geq 0\}$ and $Log_{p_2}(p_1)$ is irrational. Then in 
Lemma~\ref{lemma:EmptyLemma}
we substitute $t=k_{j+1}-k_j-1$ and we obtain a gap of size
\equ{0< p_1^{x_2(k_j)}p_2^t-p_2^{k_j+t}=p_2^{t}(p_1^{x_2(k_j)}-p_2^{k_j})\geq p_2^t=
p_2^{k_{j+1}-k_j-1}.}
Hence the limit superior of the gaps tend to infinity in the multiplicatively closed set 
$\mbb{S}$ using Lemma~\ref{lemma:IncreasingLemma}.
Now Theorem~\ref{theorem:GapsDoublyGenerated} follows.
\end{proof}
\begin{note}
Via the sequence $k_j$ we know the prime factorization of the end points of the intervals 
$(p_2^{k_j+t},p_1^{x_2(k_j)}p_2^t)$ for $0<t<k_{j+1}-k_j$,
which are all gap intervals.
\end{note}
To prove Corollary~\ref{cor:GapsDoublyGenerated} we can use 
Theorem~\ref{theorem:GapsDoublyGenerated} by observing that using Lemma~\ref{lemma:Irrational}
the pair $(p_1,p_2)$ is an irrational pair if both $p_1,p_2$ are primes, which also implies that 
both $Log_{p_1}(p_2),Log_{p_2}(p_1)$ are irrational.

Here we give an example illustrating the ideas used to prove 
Theorem~\ref{theorem:GapsDoublyGenerated}.
\begin{example}[Main example]
\label{example:Main}
Consider the irrational $\frac{1}{Log_2(3)}$. The first few terms of the sequence $k_j$, which 
is defined by the fractional parts
\equ{z_{k_j}<z_{k_{j-1}}=min\{z_1,z_2,\ldots,z_{k_j-1}\}}
is given by
\equa{&\{1,3,5,17,29,41,94,147,200,253,306,971,1636,2301,2966,3631,4296,\\
&4961,5626,6291,6956,7621,8286,8951,9616,10281,10946,11611,12276,\\
&12941,13606,14271,14936,15601,47468,79335,190537\}.}
The corresponding first few terms of the sequence $x_2(k_j)$ is given by
\equa{&\{2,5,8,27,46,65,149,233,317,401,485,1539,2593,3647,4701,5755,6809,\\
&7863,8917,9971,11025,12079,13133,14187,15241,16295,17349,18403,\\
&19457,20511,21565,22619,23673,24727,75235,125743,301994\}.}
The first few terms of the rational approximation seqence to $\ga$ is given by
\equa{&\{\frac{1}{2},\frac{3}{5},\frac{5}{8},\frac{17}{27},\frac{29}{46},\frac{41}{65},
\frac{94}{149}, \frac{147}{233},\frac{200}{317},\frac{253}{401},\frac{306}{485},
\frac{971}{1539},\frac{1636}{2593}, \frac{2301}{3647},\frac{2966}{4701},\\
&\frac{3631}{5755},\frac{4296}{6809},\frac{4961}{7863},
\frac{5626}{8917},\frac{6291}{9971},\frac{6956}{11025},\frac{7621}{12079},\frac{8286}{13133},
\frac{8951}{14187},\frac{9616}{15241},\\
&\frac{10281}{16295},\frac{10946}{17349},
\frac{11611}{18403},\frac{12276}{19457}, \frac{12941}{20511},\frac{13606}{21565},
\frac{14271}{22619},\frac{14936}{23673},
\frac{15601}{24727},\\
&\frac{47468}{75235},
\frac{79335}{125743},\frac{190537}{301994}\}.}
This seqence for approximate inverses for the fraction $\frac{190537}{301994}$ is 
given by
\equa{&\{1,2,5,8,27,46,65,149,233,317,401,485,1539,2593,3647,4701,5755,\\
&6809,7863,8917,9971,11025,12079,13133,14187,15241,16295,17349,\\
&18403,19457,20511,21565,22619,23673,24727,75235,125743,301994\}.}
We note that it matches with $x_2(k_j)$. Actually this can be obtained for any suitable 
rational approximation sequence for $\ga$. The first few gaps of intervals with the prime 
factorization of end-points of the gap intervals of the form 
\equ{(p_2^{k_{j+1}-1}\ldots p_1^{x_2(k_j)}p_2^{k_{j+1}-k_j-1})}
using this method is given by
\equa{&(3^2\ldots 2^23^1),(3^4\ldots 2^53^1),(3^{16}\ldots 2^83^{11}),
(3^{28}\ldots 2^{27}3^{11}),(3^{40}\ldots 2^{46}3^{11}),\\
&(3^{93}\ldots 2^{65}3^{52}),(3^{146}\ldots 2^{149}3^{52}),(3^{199}\ldots 2^{233}3^{52}),
(3^{252}\ldots 2^{317}3^{52}),\\
&(3^{305}\ldots 2^{401}3^{52}),(3^{970}\ldots 2^{485}3^{664}),(3^{1635}\ldots 2^{1539}3^{664}),\\
&(3^{2300}\ldots 2^{2593}3^{664}),
(3^{2965}\ldots 2^{3647}3^{664}),(3^{3630}\ldots 2^{4701}3^{664}),\\
&(3^{4295}\ldots 2^{5755}3^{664}),(3^{4960}\ldots 2^{6809}3^{664}),
(3^{5625}\ldots 2^{7863}3^{664}),\\
&(3^{6290}\ldots 2^{8917}3^{664}),
(3^{6955}\ldots 2^{9971}3^{664}),
(3^{7620}\ldots 2^{11025}3^{664}),\\
&(3^{8285}\ldots 2^{12079}3^{664}),
(3^{8950}\ldots 2^{13133}3^{664}),
(3^{9615}\ldots 2^{14187}3^{664}),\\
&(3^{10280}\ldots 2^{15241}3^{664}),
(3^{10945}\ldots 2^{16295}3^{664}),
(3^{11610}\ldots 2^{17349}3^{664}),\\
&(3^{12275}\ldots 2^{18403}3^{664}),
(3^{12940}\ldots 2^{19457}3^{664}),
(3^{13605}\ldots 2^{20511}3^{664}),\\
&(3^{14270}\ldots 2^{21565}3^{664}),
(3^{14935}\ldots 2^{22619}3^{664}),
(3^{15600}\ldots 2^{23673}3^{664}),\\
&(3^{47467}\ldots 2^{24727}3^{31866}),
(3^{79334}\ldots 2^{75235}3^{31866}),
(3^{190536}\ldots 2^{125743}3^{111201}).}
\end{example}

\section{\bf{An open question}}
\label{sec:OQ}
In this section we state an open question arising from Theorem~\ref{theorem:GapsDoublyGenerated}.

\begin{ques}
\label{ques:Gaps}
Let $\mbb{S}=\{1<a_1<a_2<\ldots<\}\subs \mbb{N}$ be a finitely generated multiplicatively closed 
infinite set generated by positive integers $d_1,d_2,\ldots,d_n$. How do we construct explicitly 
arbitrarily large integer intervals with known prime factorization of the end points,
which do not contain any elements from the set $\mbb{S}$ using the positive integer 
$d_1,d_2,\ldots,d_n?$
\end{ques}
In the later sections we consider some implications of our results regarding this open
question and partially answer this question in the affirmative.

\section{\bf{On a generalization of this method to more than two generators}}
\label{sec:GMTG}
In the proof of the main Theorem~\ref{theorem:GapsDoublyGenerated} we know the
prime factorizations of both the end points of the gap interval via the stabilization sequence.
Sometimes knowing factorizations is helpful because of the following note.
\begin{note}
If a large number $N$ has exactly has two large prime factor pair say $\{q_1,q_2\}$ and 
if $N$ lies in a gap interval of multiplicatively closed set generated by $p_1,p_2$ then 
we can postively conclude that the factor pair of $N,\{q_1,q_2\} \neq \{p_1,p_2\}$. 
The gap intervals in Theorem~\ref{theorem:GapsDoublyGenerated} are easy to generate
for any pair $\{p_1,p_2\}$ such that $Log_{p_1}(p_2)$ is irrational.
\end{note}

In this section we point out that a certain
generalization of this method of proof of Theorem~\ref{theorem:GapsDoublyGenerated} 
to more than two generators is not directly feasible.
In particular in an attempt to answer Question~\ref{ques:Gaps} we prove a lemma, which says 
that the same technique may or may not be extendable for more than two generators.
\begin{lemma}
\label{lemma:235}
Let $G=\{p_1<p_2<\ldots<p_l\}$ be a finite set of primes. Let $k$ be any positive integer. 
Consider the monoid
$T=\{\us{i=1}{\os{l-1}{\sum}} x_iLog_{p_l}p_i \mid x_i \in \mbb{N}\cup \{0\}\}$. Consider the set
$T_k=T \cap (k,k+1)$. Let $z_k = min(T_k-k)$. Let $z_{k_j}$ be a monotone decreasing sequence 
converging to zero constructed from $z_k$ defined by the property that
\equ{z_{k_j}<z_{k_{j-1}}=min\{z_1,z_2,\ldots,z_{k_j-1}\}}
then the sequence of integers $\{k_{j+1}-k_j: j \in \mbb{N}\}$ need not be increasing.
\end{lemma}
\begin{proof}
Consider the following example. Let $\{p_1=2<p_2=3<p_3=5\}$. By calculating the logarithm of 
numbers to the base $5$ in the sequence $\{2^i3^j\mid 0 \leq i,j \leq 50\}$ or by actually 
showing inequalities we obtain
\begin{itemize}
\item $k_0=0,z_{k_0}=z_0=Log_5(2)-0$.
\item $k_1=1,z_{k_1}=z_1=Log_5(2.3)-1$.
\item $k_2=2,z_{k_2}=z_2=Log_5(3^3)-2$.
\item $k_3=3,z_{k_3}=z_3=Log_5(2^7)-3$.
\item $k_4=7,z_{k_4}=z_{7}=Log_5(2^2.3^9)-7$.
\item $k_5=8,z_{k_5}=z_{8}=Log_5(2^{17}.3)-8$.
\item $k_6=13,z_{k_6}=z_{13}=Log_5(2^8.3^{14})-13$.
\item $k_7=14,z_{k_7}=z_{14}=Log_5(2^{23}.3^{6})-14$.
\end{itemize}
We can show the inequalities $z_{k_0}>z_{k_1}>z_{k_2}>z_{k_3}>z_{k_4}>z_{k_5}>z_{k_6}>z_{k_7}$
and \equa{k_1-k_0=k_2-k_1=k_3-k_2=1<k_4-k_3=4>k_5-k_4&=1<\\
k_6-k_5&=5>k_7-k_6=1,} 
which is not increasing. This proves the lemma. However we mention that it is possible that 
$\us{j \lra \infty}{limsup}(k_{j+1}-k_j)=\infty$, which additionally requires a 
proof.
\end{proof}
\section{\bf{Geometry of singly and doubly generated multiplicatively closed sets}}
\label{sec:GeometryMultClosed}
In this section we partially answer Question~\ref{ques:Gaps} using 
Theorems~\ref{theorem:DoubleGen},\ref{theorem:LineClassification} in Theorem
~\ref{theorem:MultiplyGeneratedDoublyClosedLine}. 
First we begin with a few definitions.
\begin{defn}
\label{defn:ProjSpaceMult}
Let $\mbb{Q}$ denote the field of rational numbers. Let $\mbb{Q}_{\geq 0}$ denote the set of 
non-negative rationals.
Define an equivalence relation $\sim_R$ on 
\equ{\mbb{Q}_{\geq 0}^{\infty}\bs \{0\}=\us{i=1}
{\os{\infty}{\bigoplus}}\mbb{Q}_{\geq 0}\bs \{0\}.}
We say $(a_1,a_2,\ldots,)\sim_R (b_1,b_2,\ldots,) \in \mbb{Q}_{\geq 0}^{\infty}\bs \{0\}$ if 
there exists $\gl\in \mbb{Q}^{+}$ such that $a_i=\gl b_i$ for all $i \geq 1$.
Let $\mbb{PF}^{\infty}_{\mbb{Q}_{\geq 0}}$ denote the projective space 
\equ{\mbb{PF}^{\infty}_{\mbb{Q}_{\geq 0}}=\frac{\mbb{Q}_{\geq 0}^{\infty}\bs \{0\}}{\sim_R}.}
\end{defn}
\begin{defn}
Let $\mbb{Q}$ denote the field of rational numbers. Define an equivalence relation on
\equ{\us{i\geq 1}{\bigoplus}\mbb{Q}\bs\{0\}.}
We say $(a_1,a_2,\ldots,a_n)\sim_R(b_1,b_2,\ldots,b_n)$ if $a_i=\gl b_i$ for some
$\gl\in \mbb{Q}^{*}$. Let $\mbb{PF}^{\infty}_{\mbb{Q}}$ denote the space 
\equ{\mbb{PF}^{\infty}_{\mbb{Q}}=\frac{\us{i \geq 1}{\bigoplus}\mbb{Q} \bs\{0\}}{\sim_R}.}

The space $\mbb{PF}^{\infty}_{\mbb{Q}_{\geq 0}} \subs \mbb{PF}^{\infty}_{\mbb{Q}}$ as the subset 
of points, which have all non-negative and at least one positive integer representatives. 
We note that if two finite tuples, which have positive coordinates are rational multiple of each 
other then they are positive rational multiple of each other.
\end{defn}
\begin{defn}
\label{defn:MaxMult}
Let $\mbb{P}=\{p_1=2,p_2=3,p_3=5,\ldots,\}\subs \mbb{N}$ be the set of primes, where $p_i$ 
denote the $i^{th}-$prime. We say a set $\mbb{S}\subs \mbb{N}$ is singly generated
multiplicatively closed if $\mbb{S}=\{1,f,f^2,\ldots,\}$ for some $f \in \mbb{N},f \neq 1$. 
We say $\mbb{S}$ is a singly generated maximal multiplicatively closed set if $\mbb{T}$ is any 
singly generated multiplicatively closed set and $\mbb{T} \sups \mbb{S}$ then $\mbb{T}=\mbb{S}$.
\end{defn}

\begin{defn}
\label{defn:DoublyMultClosedLine}
Let $L$ be a line obtained by joining two points 
$P_1,P_2 \in \mbb{PF}^{\infty}_{\mbb{Q}_{\geq 0}} \subs \mbb{PF}^{\infty}_{\mbb{Q}}$.
We say $L$ is a doubly multiplicatively closed line, if we consider only integers and 
(not elements of $\mbb{Q}_{\geq 0}\bs\Z_{\geq 0}$) associated to all tuples whose equivalence
classes are points that lie on $L$ (refer to the proof of Theorem~\ref{theorem:Bijection}) then it
gives rise to a doubly generated multiplicatively closed set. 

In view of Example~\ref{example:LineNotMultClosed} not all lines $L$ are
doubly multiplicatively closed lines. However we note that each point 
$P\in \mbb{PF}^{\infty}_{\mbb{Q}_{\geq 0}}$ gives rise to a unique maximal singly generated 
multiplicatively closed set (See Theorem~\ref{theorem:Bijection}).
\end{defn}

Now we state a correspondence theorem.
\begin{theorem}
\label{theorem:Bijection}
Let \equa{\mcl{S}=\{\mbb{S}\subs \mbb{N}\mid& \text{ such that }\mbb{S} 
\text{ is a maximal singly generated multiplicatively}\\
&\text{closed set.}\}}
Then there is a bijective correspondence between
\equ{\mcl{S} \llra \mbb{PF}^{\infty}_{\mbb{Q}_{\geq 0}}}
i.e. between maximal singly generated multiplicatively closed sets and the points of the space 
$\mbb{PF}^{\infty}_{\mbb{Q}_{\geq 0}}$
given by 
\equa{\mbb{S}&=\{f^n\mid 0 \leq n\in \mbb{N}\cup\{0\},\\
f&=\us{j=1}{\os{k}{\prod}}p_{i_j}^{r_{i_j}} 
\text{ with } p_{i_1}<p_{i_2}<\ldots<p_{i_k} \in \mbb{P},r_{i_1},\ldots,r_{i_k}\in \mbb{N}\}\\
\lra P&=[\ldots:r_{i_1}:\ldots:r_{i_2}:\ldots:\ldots:\ldots:r_{i_k}:\ldots] \in 
\mbb{PF}^{\infty}_{\mbb{Q}_{\geq 0}}}
where the coordinates of any point in $\mbb{PF}^{\infty}_{\mbb{Q}_{\geq 0}}$ are ordered 
according to increasing sequence of primes in the set $\mbb{P}$.
\end{theorem}
\begin{proof}
The bijection is given as follows. Let \equ{\mbb{S}=\{1,f,f^2,\ldots\}} be any singly generated 
multiplicatively closed set. Let
\equ{f=\us{j=1}{\os{k}{\prod}}p_{i_j}^{r_{i_j}} 
\text{ with } p_{i_1}<p_{i_2}<\ldots<p_{i_k} \in \mbb{P},r_{i_1},\ldots,r_{i_k}\in \mbb{N}.}
To this multiplicatively closed set we associate the point \equ{P=[\ldots:r_{i_1}:\ldots:
r_{i_2}:\ldots:\ldots:\ldots:r_{i_k}:\ldots] \in \mbb{PF}^{\infty}_{\mbb{Q}_{\geq 0}}.}
The condition that $\mbb{S}$ is maximal is equivalent to the condition
\equ{gcd(r_{i_1},r_{i_2},\ldots,r_{i_k})=1.} Also given any point $P$ in 
$\mbb{PF}^{\infty}_{\mbb{Q}_{\geq 0}}$
there is a unique non-negative integer coordinate representative of $P$ with $gcd$ of the 
coordinates equal to one, which gives rise to the integer $f \in \mbb{N}$ with $f \neq 1$.

This establishes the bijection and hence Theorem~\ref{theorem:Bijection} follows.
\end{proof}
\begin{theorem}[Log-Rationality]
\label{theorem:Equivalence}
Let $P_1,P_2$ be two points (possibly the same point) in $\mbb{PF}^{\infty}_{\mbb{Q}_{\geq 0}}$.
Let \equ{g_1=\us{j=1}{\os{t}{\prod}}p_{i_j}^{r_{i_j}},
g_2=\us{j=1}{\os{u}{\prod}}q_{i_j}^{s_{i_j}}}  
two positive integers $(>1)$ with their unique prime factorizations such that
\equa{P_1=[\ldots:r_{i_1}:\ldots:r_{i_2}:\ldots:\ldots:\ldots:r_{i_t}:\ldots]\\
P_2=[\ldots:s_{i_1}:\ldots:s_{i_2}:\ldots:\ldots:\ldots:s_{i_u}:\ldots].}
Let $f_1,f_2$ be the corresponding positive integers $(>1)$ under the bijection given in
Theorem~\ref{theorem:Bijection} then the following are equivalent.
\begin{enumerate}
\item (Log-Rationality:) $Log_{g_1}(g_2)$ is rational.
\item $P_1=P_2$
\item $f_1=f_2$.
\item The multiplicatively closed set $\mbb{T}=\{g_1^{i}g_2^{j}\mid i,j \geq 0\}$ is contained 
in a singly generated maximal multiplicatively closed set.
\end{enumerate}
\end{theorem}
\begin{proof}
Suppose $Log_{g_1}(g_2)=\frac mn$ is rational. Then we have $g_2^n=g_1^m$. So the distinct prime 
factors of $g_1,g_2$ agree and we also have that their exponents are projectively equivalent. 
Hence we get $P_1=P_2$. So this implies $f_1=f_2=f$ say. Then we get that 
$\mbb{T}\subs \{1,f,f^2,\ldots,\}$.

For the converse if $\mbb{T}\subs \{1,f,f^2,\ldots,\}$ for some $1 \neq f \in \mbb{N}$ then 
$g_1=f^n$, $g_2=f^m$ and we have $g_2^n=g_1^m$. Hence $Log_{g_1}(g_2)=\frac mn$ is rational. 
This completes the equivalence of the statements $(1),(2),(3),(4)$ and also proves 
Theorem~\ref{theorem:Bijection}.
\end{proof}
Now we have the following corollary.
\begin{cor}
\label{cor:Equivalence}
A multiplicatively closed set $\mbb{T}=\{g_1^ig_2^j\mid i,j\geq 0,g_1,g_2 \in \mbb{N}\bs\{1\}\}$ 
is not contained in a singly generated multiplicatively closed set if and only if 
$Log_{g_1}(g_2),Log_{g_2}(g_1)$ are both irrational if and only if $g_1,g_2$ represent two 
distinct points in the projective space
$\mbb{PF}^{\infty}_{\mbb{Q}_{\geq 0}}$.
\end{cor}
In the theorem that follows we give a criterion as to when a multiplicatively closed set is 
contained in a doubly generated multiplicatively closed set.
\begin{theorem}
\label{theorem:DoubleGen}
Let $\mbb{S}=\{g_1^{i_1}g_2^{i_2}\ldots g_r^{i_r}\mid i_1,i_2,\ldots,i_r \in \mbb{N}\cup \{0\}\}$
be a multiplicatively closed set generated by $r-$elements.
Suppose corresponding to these positive integers $g_i:1 \leq i \leq r$ the points 
$[g_i]: 1\leq i \leq r \in \mbb{PF}^{\infty}_{\mbb{Q}_{\geq 0}} \subs 
\mbb{PF}^{\infty}_{\mbb{Q}}$ lie on a projective line $L$ (i.e. a rank$=1,2$ 
condition on the matrix of powers of primes in the prime factorizations of $g_i$) obtained by
joining two points of $\mbb{PF}^{\infty}_{\mbb{Q}_{\geq 0}}$ whose corresponding integers are 
relatively prime. Then $\mbb{S}$ is contained in a doubly generated multiplicatively closed set.
\end{theorem}
\begin{proof}
If $\mbb{S}$ gives rise to a single point then there is nothing to prove. So
let $P_1,P_2 \in \mbb{PF}^{\infty}_{\mbb{Q}_{\geq 0}}$ be any two distinct points, which gives 
rise to the projective line $L$.
Let $p_1,p_2$ be the positive integers, which represent these points $P_1,P_2$ with 
$gcd(p_1,p_2)=1$. Then the hypothesis that the points $[g_i]$ lie on the projective line 
$P_1P_2$ implies that there exists integers $a_i,b_i,c_i \geq 0$  such that 
$p_1^{a_i}p_2^{b_i}=g_i^{c_i}$ for $1 \leq i\leq r$.
Consider the unique prime factorization of 
\equ{p_1=q_1^{s_1}q_2^{s_2}\ldots q_l^{s_l},p_2=q_1^{t_1}q_2^{t_2}\ldots q_l^{t_l},}
where we assume without loss of generality that $gcd(s_1,s_2,\ldots,s_l)=1$,
$gcd(t_1,t_2,\ldots,t_l)=1$. If in addition we have $gcd(p_1,p_2)=1$
then we have $s_jt_j=0: 1 \leq j \leq l$ but one of $s_j$ and $t_j$ is non-zero for each $j$. 
In all cases we conclude that $c_i\mid s_ja_i, c_i \mid t_jb_i$ for $1 \leq j \leq l$. So 
$c_i \mid a_i,c_i\mid b_i : 1 \leq i \leq r$ as $gcd(s_1,s_2,\ldots,s_l)=1,
gcd(t_1,t_2,\ldots,t_l)=1$.  Hence the set $\mbb{T}=\{p_1^ip_2^j\mid i,j \geq 0\} \sups \mbb{S}$
and this proves Theorem~\ref{theorem:DoubleGen}.
\end{proof}
\begin{example}
\label{example:LineNotMultClosed}
Let $g_1=45,g_2=20,g_3=30$. Then we have $g_1g_2=g_3^2$. So the doubly generated 
multiplicatively closed set generated by $g_1,g_2$ contains $g_3^2$ but not $g_3$.
However there is no doubly generated multiplicatively closed set containing all $g_1,g_2,g_3$ 
because there are no two distinct non-trivial common factors of
$g_1,g_2,g_3$ as $gcd(g_1,g_2,g_3)=5$, which is prime. Now the corresponding exponents satisfy
\equ{(0,2,1,0,\ldots)+(2,0,1,0,\ldots)=2.(1,1,1,0,\ldots).}
since $g_1g_2=g_3^2$ and the exponent vectors lie on a projective line 
$L \subs \mbb{PF}^{\infty}_{\mbb{Q}}$.
\end{example}
\begin{note}
Theorem~\ref{theorem:DoubleGen} can be generalized as follows. Let 
\equ{\mbb{S}=\{g_1^{i_1}g_2^{i_2}\ldots g_r^{i_r}\mid i_1,i_2,\ldots,i_r \in \mbb{N}\cup \{0\}\}}
be a multiplicatively closed set generated by $r-$ elements.
Fix a prime $p=2$. Suppose there exists two positive integers $p_1,p_2$ such that the monoid 
$\{aLog_{p}(p_1)+bLog_{p}(p_2)\mid a,b \in \Z_{\geq 0}\}$ contains the set 
$\{Log_{p}(g_i):1 \leq i \leq r\}$ then the set 
$\mbb{S} \subs \mbb{T}=\{p_1^ip_2^j\mid i,j\geq 0\}$.
\end{note}

\begin{note}
\label{note:ExampleLine}
In Example~\ref{example:LineNotMultClosed}, for all integer representatives $f \in \mbb{N}$ such 
that $[f] \in L$ we have $5 \mid f$. This is the only prime with this property for the line $L$,
which is not a doubly multiplicatively closed line. So does there exist a doubly 
multiplicatively closed line with such a prime? Definitely not when there are only two primes 
involved with the line $L$.

In Example~\ref{example:LineNotMultClosed}, we have the following properties holding true.
\begin{itemize}
\item For all integer representatives $f \in \mbb{N}$ such that $[f] \in L$ we have $5 \mid f$ 
and this is the only such prime. Neither of the primes
$2,3$ satisfy this property.
\item There exists numbers $g_1=45,g_2=20$ whose points lie on $L$ and two primes $2,3$ such 
that \equ{3 \mid 45, 3\nmid 20,2 \mid 20, 2 \nmid 45.}
\item The lattice $M$ corresponding to $L$ is a two dimensional lattice, which does not possess a 
basis $\{x,y\}$ such that $M \cap \Z^r_{\geq 0}$ satisfies
the monoid addition property. i.e.
\equ{\boxed{ax+by \in M\cap \Z^r_{\geq 0} \Lra a \geq 0,b \geq 0.}}
\end{itemize}
\end{note}

In the following theorem we classify doubly multiplicatively closed lines.
\begin{theorem}
\label{theorem:LineClassification}
A line $L$ joining two points $P_1,P_2 \in \mbb{PF}^{\infty}_{\mbb{Q}_{\geq 0}} \subs 
\mbb{PF}^{\infty}_{\mbb{Q}}$ is a doubly multiplicatively closed line
if and only if there exists two points $Q_1=[q_1],Q_2=[q_2]\in 
\mbb{PF}^{\infty}_{\mbb{Q}_{\geq 0}} \subs \mbb{PF}^{\infty}_{\mbb{Q}}$ with positive integers
\equ{q_1=p_1^{a_1}p_2^{a_2}\ldots p_r^{a_r},q_2=p_1^{b_1}p_2^{b_2}\ldots p_r^{b_r}.}
with the following properties.
\begin{enumerate}
\item Trivial Index Property(Alternative, refer~\ref{item:TIV}): The gcd of two by two minors of 
\equ{\mattwofour {a_1}{a_2}{\ldots}{a_r}{b_1}{b_2}{\ldots}{b_r}} is one.
\item Monoid Addition Property(Alternative, refer~\ref{item:MAP}): There exists two subscripts 
$i,j$ such that $a_ib_i=0=a_jb_j$ and either $a_ib_j\neq 0$ or $a_jb_i \neq 0$.
\end{enumerate}
In particular if there exists two points $Q_1,Q_2 \in L\cap 
\mbb{PF}^{\infty}_{\mbb{Q}_{\geq 0}}$ such that their corresponding integer representatives
are relatively prime then $L$ is a doubly multiplicatively closed line.
\end{theorem}
\begin{proof}
First we prove the last assertion. Suppose there exists such points $Q_1,Q_2$ on $L$ and let 
$q_1,q_2$ be the corresponding integers. Let \equ{q_1=p_1^{s_1}p_2^{s_2}\ldots p_l^{s_l},
q_2=p_{l+1}^{s_{l+1}}p_{l+2}^{s_{l+2}}\ldots p_n^{s_n}.} be their unique prime factorizations
with $s_i\in \mbb{N}:1 \leq i\leq n,gcd(q_1,q_2)=1$. Now we choose $q_1,q_2$ such that 
\equ{gcd(s_1,s_2,\ldots,s_l)=1=gcd(s_{l+1},s_{l+2},\ldots,s_n).}

If $g\in \mbb{N}$ such that $[g]\in L$ then there exist $a,b,c \in \mbb{N}$ such that 
$q_1^aq_2^b=g^c$. This implies
\equ{c \mid as_i:1 \leq i \leq l,c \mid bs_i: l+1\leq i\leq n \Ra c \mid a,c\mid b.}
So $g \in \mbb{T}=\{q_1^iq_2^j\mid i,j\in \mbb{N}\cup \{0\}\}$.
In this particular case we also have in the matrix
\equ{\mattwosix {s_1}{\ldots}{s_l}{0}{\ldots}{0}{0}{\ldots}{0}{s_{l+1}}{\ldots}{s_n}}
has the property that its gcd of two minors equal 
$gcd(s_is_j:1 \leq i \leq l,l+1 \leq j \leq n)=1$. This proves the last assertion.

Now we prove the first assertion. In particular the implication $(\La)$.
Assume $L$ has such points $Q_1=[q_1],Q_2=[q_2]$. 

Let $V=\mbb{Q}-span$ of the vectors $\{a=(a_1,\ldots,a_r),(b_1,\ldots,b_r)\}$.
Let $M=\Z-span$ of the vectors $\{a,b\}$.
Then we have the following properties.
\begin{enumerate}[label=(\Alph*)]
\item \label{item:TIV}
Trivial Index Property: \equ{V \cap \Z^r=M.} This follows because of the gcd of the 
$2\times 2$ minors is one. i.e. $M$ has a trivial index in $V \cap \Z^r$.
Using the theorem for sublattices we get that for the tower of sublattices
\equ{M \subs V \cap \Z^r \subs \Z^r}
that there exists a basis of $\Z^r$ given by $\{u_1,u_2,\ldots,u_r\}$ and positive integers 
$d_1,d_2$ such that $\{u_1,u_2\}$ is a basis of $V \cap \Z^r$
and $\{d_1u_1,d_2u_2\}$ is a basis of $M$, which has therefore index $d_1d_2$, which is also 
gcd of the $2 \times 2$ minors of $\{d_1u_1,d_2u_2\}$.
Since $\{d_1u_1,d_2u_2\}$ differ from the basis $\{a,b\}$ of $M$ by an $SL_2(\Z)$ matrix we 
have $d_1d_2=1$.
\item \label{item:MAP}
Monoid Addition Property: 
\equ{\ga a+ \gb b \in M\cap \Z^{r}_{\geq 0} \Lra \ga \geq 0,\gb \geq 0.} This follows because
there exists two subscripts $i,j$ such that $a_ib_i=0=a_jb_j$ and either $a_ib_j\neq 0$ or 
$a_jb_i \neq 0$ and the coordinate entries of both $a,b$ are non-negative.
\end{enumerate}
So we get that if $g \in \mbb{N}$ such that $[g]\in L$ then $g=q_1^iq_2^j$ for some 
$i,j\in \mbb{N}\cup \{0\}$. So the required multiplicatively closed
set representing the line is $\mbb{T}=\{q_1^iq_2^j\mid i,j \in \mbb{N}\cup \{0\}\}$.

Now we prove the implication $(\Ra)$. 

Suppose $L$ is a multiplicatively closed line with the 
multiplicatively closed set being
$\mbb{T}=\{q_1^iq_2^j\mid i,j \in \mbb{N}\cup \{0\}\}$. So if $g \in \mbb{N}$ such that 
$[g]\in L$ then $g \in \mbb{T}$. Let the prime exponent vectors
of $q_1,q_2$ be $s,t$ with $s=(s_1,s_2,\ldots,s_r),t=(t_1,t_2,\ldots,t_r),
gcd(s_1,s_2,\ldots,s_r)=1=gcd(t_1,t_2,\ldots,t_r)$.
This proof is a bit long. We prove both the Trivial Index Property and the Monoid 
Addition Property for $\{s,t\}$.
\begin{claim}
\label{claim:BasisNonNegative}
Let $V$ be a two dimensional $\mbb{Q}-$vector space spanned by $s=(s_1,s_2,\ldots,s_r),
t=(t_1,t_2,\ldots,t_r)$. Let $M=V \cap \Z^r$. Then there exists $w\in M$ with all its coordinate entries non-negative such that $\{s,w\}$ is a basis for $M$. 
\end{claim}
\begin{proof}[Proof of Claim]
We observe that $V$ is the corresponding affine space defined by the projective line $L$ and 
also that $M=V \cap \Z^r$ is a two dimensional
lattice. Now by a theorem on sublattices of $\Z^r$ it follows that there exists a basis of 
$\Z^r$ say \equ{\{u=(u_1,u_2,\ldots,u_r),v=(v_1,v_2,\ldots,v_r),w^1,w^2,\ldots,w^{r-2}\}}
and positive integers $d_1,d_2$ such that $\{d_1u,d_2v\}$ is a basis of $M$ with 
$d_1 \mid d_2$. Since $M$ contains a gcd one vector on every line we have 
$d_1=1=d_2$.
If $\ga u+\gb v=s$ then $gcd(\ga,\gb)=1$ because $(\ga)+(\gb)=(\ga u_1,\ga u_2,\ldots,
\ga u_r)+(\gb v_1,\gb v_2,\ldots,\gb v_r)\sups (\ga u_1+\gb v_1,
\ga u_2+\gb v_2,\ldots,\ga u_r+\gb v_r)=(s_1,s_2,\ldots,s_r)=\Z$. 

Hence there exists 
$\mattwo {\ga}{\gb}{\gga}{\gd}\in SL_2(\Z)$ such that
\equ{\mattwo {\mattwo {\ga}{\gb}{\gga}{\gd}}{0_{2 \times (r-2)}}{0_{(r-2) \times 2}}{I_{(r-2)
\times (r-2)}}\matcolfive {u}{v}{w^1}{\vdots}{w^{r-2}}=
\matcolfive {s}{w}{w^1}{\vdots}{w^{r-2}},} where $w=\gga u +\gd v$. Apriori $w$ need not have 
non-negative entries.
Using a unipotent lower triangular matrix over $\Z$ we need to consider only those entries 
of $w$ whose corresponding entries in $s$ are zero.
Now the vector $t$, which has non-negative entries lies in the span $M$ of $s,w$. 
i.e $t=\gep s+\gm w$ with $\gep \in \mbb{Z},\gm \in \mbb{Z}^{*}=\Z\bs\{0\}$.
Now if $s_i=0,w_i \neq 0$ then $sign(w_i)= sign(\gm)$.
If $sign(\gm)$ is negative then we consider $-w$ instead of $w$. Then we get that the $w_i$ has 
non-negative sign whenever $s_i$ is zero. Now again using unipotent lower triangular matrix
over $\Z$ we make the sign of the remaining entries of $w$ non-negative. Hence we arrive at a 
basis $\{s,w\}$ of $M$ such that both have non-negative integer entries. 
We have obtained $\{s,w\}$ from $\{u,v\}$ by an $\ti{S}L_2(\Z)$ transformation with determinant 
$\pm 1$.
This proves Claim~\ref{claim:BasisNonNegative}
\end{proof}
\begin{claim}[Trivial index property]
\label{claim:TrivialIndexProperty}
Let $\{s,w\}$ be the basis of $M$ obtained from Claim~\ref{claim:BasisNonNegative}.
Let \equ{q_1=p_1^{s_1}p_2^{s_2}\ldots p_r^{s_r},e=p_1^{w_1}q_2^{w_2}\ldots p_r^{w_r}.}
Since $M \cap \Z^r_{\geq 0}$ corresponds to a doubly multiplicatively closed set 
$\mbb{T}=\{q_1^iq_2^j\mid i,j \in \mbb{N}\cup \{0\}\}$
we have
\begin{itemize}
\item \equ{\{q_1^ie^j\mid i,j \in \Z\}=\{q_1^iq_2^j\mid i,j\in \Z\}.}
\item The $\Z-span$ of $\{s,t\}$ is the same as $\Z-span$ of $\{s,w\}$.
\item If for some $\ga,\gb \in \mbb{Q}, \ga s+\gb t \in \Z^r$ then $\ga,\gb \in \Z$.
\end{itemize}
\end{claim}
\begin{proof}[Proof of Claim]
Since $q_1,e$ corresponds to points in $M \cap \Z^r_{\geq 0}$ we have $e=q_1^{m_1}q_2^{m_2},
m_1,m_2\in \mbb{N}\cup \{0\}$.
So we have
$\{q_1^ie^j\mid i,j \in \Z\}\subs \{q_1^iq_2^j\mid i,j\in \Z\}.$ The other way containment is 
immediate since $\Z-$span of $\{s,w\}$ contains $\Z-$span of $\{s,t\}$.
Now the rest of the claim for exponents follows as $\{s,w\}$ is a $\Z-$basis for $M=V\cap \Z^r$ 
and $\mbb{Q}-$basis for $V$. This proves the trivial index property for $\{s,t\}$.
\end{proof}
\begin{claim}[Monoid addition property]
\label{claim:MonoidAdditionProperty}
The basis $\{s,t\}$ has monoid addition property.
\end{claim}
\begin{proof}[Proof of Claim]
Now suppose if all the coordinate entries of $s$ is positive. Then for some large $m\in \mbb{N}$
we have $ms-t$ has non-negative entries, which is a contradiction.
Hence there exist a subscript $i$ such that $s_i=0,t_i\neq 0$. Similarly there exist a 
subscript $j$ such that $t_j=0,s_i\neq 0$.
This proves the monoid addition property that
\equ{\ga s+ \gb t \in M\cap \Z^{r}_{\geq 0} \Lra \ga \geq 0,\gb \geq 0.}
\end{proof}
This completes the proof of Theorem~\ref{theorem:LineClassification}.
\end{proof}
\begin{example}
\label{example:LineMultClosed}
Let $g_1=10$,$g_2=15$. Then the line joining the points $[g_1],[g_2]$ is a multiplicatively 
closed line using Theorem~\ref{theorem:LineClassification},
where as the line in Example~\ref{example:LineNotMultClosed} joining 
$[\ti{g}_1=20],[\ti{g}_2=45]$ is not multiplicatively closed.
Now we could also use Theorem~\ref{theorem:LineClassification} to prove this fact in another way.
\end{example}
\begin{theorem}
\label{theorem:MultiplyGeneratedDoublyClosedLine}
Let $S\subs \mbb{N}$ be a finitely generated multiplicatively closed set whose corresponding
points lie on a double multiplicatively closed line $L$ containing points $Q_1=[q_1],Q_2=[q_2]$
satisfying trivial index property and monoid addition property then 
we have explicit expressions for the end points of certain arbitrarily large gap intervals
in the set $\mbb{S}$ using the generators $q_1,q_2$. 
\end{theorem}
\begin{proof}
This theorem follows because the set $\mbb{S}\subs \{q_1^iq_2^j\mid i,j\in \mbb{N}\cup\{0\}\}$
and then we use main result~\ref{theorem:GapsDoublyGenerated}.
\end{proof}

\section{\bf{Appendix}}
\label{sec:Appendix}
In this appendix section we prove some interesting lemmas about gaps, also present some 
motivating examples and give another constructive proof and discuss advantages
and disadvantages with respect to the above given constructive proof.

We begin with a lemma.
\begin{lemma}
\begin{enumerate}
\item Let $S \subs \mbb{N}$ be an infinite set. If \equ{\us{n \lra
\infty}{liminf}\frac{\#(S \cap [1,\ldots, n])}{n}=0} then there are
arbitrarily large gaps in $S$.
\item Let $S_i \subs \mbb{N}:1 \leq i\leq k$ be $k-$infinite subsets. If
for each $1 \leq i \leq k$ \equ{\us{n \lra \infty}{lim}\frac{\#(S_i
\cap [1,\ldots, n])}{n}=0} then there are arbitrarily large gaps in
$S=\us{i=1}{\os{k}{\bigcup}}S_i$.
\end{enumerate}
\end{lemma}
\begin{proof}
To prove $(1)$ we observe that if the gaps were bounded then 
\equ{\us{n \lra \infty}{liminf}\frac{\#(S\cap [1,\ldots, n])}{n}>0.}
To prove $(2)$ we have 
\equ{0 \leq \us{n \lra \infty}{lim}\frac{\#(S \cap[1,\ldots, n])}{n} \leq 
\us{n \lra\infty}{lim}\us{i=1}{\os{k}{\sum}}\frac{\#(S_i \cap [1,\ldots,n])}{n}=0.}
Hence using $(1)$ the gaps in $S$ is unbounded.
\end{proof}
\begin{example}
The following sets have arbitrarily large gaps.
\begin{itemize}
\item A multiplicatively closed set generated by finitely many positive integers $>1$.
\item The set of all integers, which have exactly $k-$prime factors.
\item The set of all integers, which have atmost $k-$prime factors.
\end{itemize}
\end{example}

\begin{theorem}
Let $\mbb{S}_1,\mbb{S}_2$ be two infinite subsets of $\mbb{N}$. Let 
$\mbb{S}_3=\mbb{S}_1\cup \mbb{S}_2,\mbb{S}_4=\mbb{S}_1\mbb{S}_2=
\{s_1s_2\mid s_i\in \mbb{S}_i,i=1,2\}$. Let
$\mbb{S}_i=\{1<a_{i1}<a_{i2}<\ldots\}$ for $i=1,2,3,4$. Then
\begin{enumerate}
\item $\us{j\lra \infty}{limsup} (a_{i(j+1)}-a_{ij})=\infty$ for $i=1,2 \not\Ra 
\us{j\lra \infty}{limsup} (a_{i(j+1)}-a_{ij})=\infty$ for $i=3,4$.
\item $\us{j\lra \infty}{lim} (a_{1(j+1)}-a_{1j})=\infty, \us{j\lra \infty}{limsup} 
(a_{2(j+1)}-a_{2j})=\infty$
then \linebreak $\us{j\lra \infty}{limsup} (a_{3(j+1)}-a_{3j})=\infty$ and does not imply
$\us{j\lra \infty}{limsup} (a_{4(j+1)}-a_{4j})=\infty$.
\item $\us{j\lra \infty}{lim} (a_{i(j+1)}-a_{ij})=\infty$ for 
$i=1,2 \Ra \us{j\lra \infty}{limsup}(a_{4(j+1)}-a_{4j})=\infty$.
\end{enumerate}
\end{theorem}
\begin{proof}
Let us prove $(1)$ by giving a counter example.
\begin{itemize}
\item Consider the set of natural numbers $\mbb{N}$. Decompose $\mbb{N}$ into two sets 
$\mbb{S}_1,\mbb{S}_2$
as follows. Keep the first element of $\mbb{N}$ in $\mbb{S}_1$. The next two elements in 
$\mbb{S}_2$. The next three
elements in $\mbb{S}_1$ and so on i.e.
\equa{\mbb{S}_1 &=\us{i\geq 0}{\bigcup}\{(2i+1)(i+1)-2i,\ldots,(2i+1)(i+1)\}\\
\mbb{S}_2 &=\us{i\geq 1}{\bigcup}\{i(2i+1)-2i+1,\ldots,i(2i+1)\}}
Then $\mbb{S}_1 \cup \mbb{S}_2=\mbb{N}$.
\item Partion the set of primes $\mbb{P}$ into two infinite subsets of primes 
$\mbb{PP}_1,\mbb{PP}_2$.
Let $\mbb{S}_i$ be the multiplicatively closed set generated by $\mbb{PP}_i$ for $i=1,2$. Then
$\mbb{S}_1\mbb{S}_2=\mbb{N}$ and  $\us{j\lra \infty}{limsup} (a_{i(j+1)}-a_{ij})=\infty$ for 
$i=1,2$
by an application of chinese remainder theorem.
\end{itemize}
Let us prove $(2)$. Given any $N>0$ there exists $M$ such that $a_{1(k+1)}-a_{1k}>N$ for all
$k>M$ and there exists infintely many $l>M$ such that $a_{2(l+1)}-a_{2l}>N$. Also choose large 
enough $l=l_0>M$ such that if $a_{1k_0}>a_{2l_0}$ then $k_0>M$. If $a_{2l_0}<a_{2(l_0+1)}$ are
consecutive in $\mbb{S}_1 \cup \mbb{S}_2$ then we have produced a gap more than $N$. If 
$a_{2l_0}<a_{1k_0}$ are consecutive then
\begin{itemize}
\item We have either $a_{2l_0}<a_{1k_0}<a_{1(k_0+1)}$ as consecutive integers in 
$\mbb{S}_1 \cup\mbb{S}_2$.
\item Or $a_{2l_0}<a_{1k_0}<a_{2(l_0+1)}$ as consecutive integers in $\mbb{S}_1 \cup \mbb{S}_2$.
\end{itemize}
In the first case we are done again. In the second case we have either 
$a_{1k_0}-a_{2l_0} > \frac N2,a_{2(l_0+1)}-a_{1k_0}> \frac N2$. Hence we have
produced a gap more than $\frac N2$. Moreover these gaps can be produced arbitrary number of 
times by choosing $M$ larger and larger for any positive integer $N$. So we
have $\us{j\lra \infty}{limsup} (a_{3(j+1)}-a_{3j})=\infty$.

Now for second part of $(2)$ we give a counter example.
Let $\mbb{S}_1=\{n^2 \mid n \in \mbb{N}\}$. Let $\mbb{S}_2=\{n\in \mbb{N}\mid n 
\text{ is square free}\}$.
Then $\mbb{S}_1\mbb{S}_2=\mbb{N}$. We have $\us{j\lra \infty}{lim} (a_{1(j+1)}-a_{1j})=\infty$. 
Also by an application of chinese remainder theorem we have 
$\us{j\lra \infty}{limsup} (a_{2(j+1)}-a_{2j})=\infty$.

Let us prove $(3)$. Fix a large integer $K$. Let
$\mbb{T}_1=\{1<a_{11}<a_{12}<\ldots<a_{1N}\},\mbb{T}_2=\{1<a_{21}<a_{22}<\ldots<a_{2M}\}$.
Suppose $a_{1(t+1)}-a_{1t}\geq K$ for all $t \geq N-1$ and
$a_{2(t+1)}-a_{2t} \geq K$ for all $t \geq M-1$. Let $a_{1N}a_{2M},a_{1\ti{N}}a_{2\ti{M}}$ be 
two successive numbers in the set $\mbb{S}_1\mbb{S}_2$. Then we have either $\ti{N}>N$ or
$\ti{M}>M$. We note that for $\ti{N}>N$ we have
\equ{a_{1\ti{N}}a_{2\ti{M}}-a_{1N}a_{2M} \geq
(a_{1\ti{N}}-a_{1N})a_{2\ti{M}} \geq K\text{ if }\ti{M} \geq  M.} For
\equ{a_{1\ti{N}}a_{2\ti{M}}-a_{1N}b_{2M} \geq (a_{1\ti{N}}-a_{1N})a_{2M}
\geq K\text{ if }M > \ti{M}.}
The argument is similar if $\ti{M}>M$. This holds for any large $K$. So
$\us{j\lra \infty}{limsup} (a_{4(j+1)}-a_{4j})=\infty$.

Hence we have completed the proof of this theorem.
\end{proof}
\begin{theorem}
Let $\mbb{S}_i:1 \leq i \leq n$ be finitely many infinite subsets of $\mbb{N}$. Let 
$\mbb{S}_{n+1}=\us{i=1}{\os{n}{\bigcup}}\mbb{S}_i,
\mbb{S}_{n+2}=\us{i=1}{\os{n}{\prod}}\mbb{S}_i=\{s_1s_2\ldots s_n\mid s_i
\in \mbb{S}_i,1 \leq i \leq n\}$. Let $\mbb{S}_i=\{1<a_{i1}<a_{i2}<\ldots\}:1 \leq i \leq n+2$.
If $\us{j\lra \infty}{lim} (a_{i(j+1)}-a_{ij})=\infty$ for $i=1,\ldots,n$ then
$\us{j\lra \infty}{limsup} (a_{i(j+1)}-a_{ij})=\infty$ for $i=n+1,n+2$.
\end{theorem}
\begin{proof}
The proof of this theorem is left to the interested reader.
\end{proof}
\begin{cor}
\begin{enumerate}
\item  The set of natural numbers $\mbb{N}$ cannot be written as a finite
product of sets $\mbb{S}_1\mbb{S}_2\ldots\mbb{S}_n,$ where the gaps in $\mbb{S}_i$ diverges to 
$\infty$ for $1 \leq i\leq n$.
\item The set of natural numbers $\mbb{N}$ cannot be written as a finite
union of sets $\mbb{S}_1\cup \mbb{S}_2\cup \ldots\cup \mbb{S}_n,$ where the gaps in $\mbb{S}_i$ 
diverges to $\infty$ for $1 \leq i\leq n$.
\item The multiplicatively closed subset $\mbb{S}$ of $\mbb{N}$ generated by finitely many
positive integers $>1$ has arbitrarily large gaps.
\end{enumerate}
\end{cor}
\begin{theorem}[Another constructive proof]
\label{theorem:AnotherConstructiveProof}
The multiplicatively closed subset of $\mbb{N}$ generated by finitely many positive integers 
$\mbb{S}$ has arbitrarily large gaps.
\end{theorem}
\begin{proof}
We give here another constructive proof in this Theorem. Let $K$ be
an arbitrary positive integer. Let $n_1,n_2,\ldots,n_k$ be the
generators of the multiplicatively closed set.

Define $\lceil Log_{n_i}(K) \rceil=a_i$. Then we have for all
\equ{t_i \geq a_i,t_i\in
\mbb{N},n_i^{t_i+1}-n_i^{t_i}=n_i^{t_i}(n_i-1) \geq n_i^{t_i} \geq
K.} The gap between $n_1^{t_1}n_2^{t_2}\ldots n_{k}^{t_k}$ and the
next number $l$ in the set $\mbb{S}_1\mbb{S}_2\ldots\mbb{S}_k$ is at
least $K$. Let $l=n_1^{s_1}n_2^{s_2}\ldots n_k^{s_k}$ be the next
number. Then there is at least one $i=i_0$ such that $s_i>t_i$.

Let $a=\frac{n_1^{s_1}n_2^{s_2}\ldots
n_k^{s_k}}{n_i^{s_i}},b=\frac{n_1^{t_1}n_2^{t_2}\ldots n_{k}^{t_k}}{n_i^{t_i}}$.
So we get that $n_1^{s_1}n_2^{s_2}\ldots
n_k^{s_k}-$\linebreak $n_1^{t_1}n_2^{t_2}\ldots n_{k}^{t_k}=
n_i^{s_i}a-n_i^{t_i}b=n_i^{t_i}(n_i^{s_i-t_i}a-b) \geq
n_i^{t_i} \geq K$.
\end{proof}
\begin{note}
The difference between this constructive proof and the other
constructive proof is that we do not exactly know the right end
point $l$ of this Gap-Interval as we do not know its prime
factorization exactly. However we were able to locate a point
$n_1^{a_1}n_2^{a_2}\ldots n_{k}^{a_k}$ and a gap interval of size at
least $K$ with this integer as the left end point for every positive
integer $K>0$.
\end{note}

\end{document}